\documentclass{article}
\pdfoutput=1
\usepackage{fullpage}

\usepackage{graphicx}
\usepackage{hyperref}

\usepackage[utf8]{inputenc}

\usepackage[T1]{fontenc}
\usepackage{fullpage}

\usepackage{amsmath,amssymb,amsthm}
\usepackage{algorithm}
\usepackage{algorithmic}

\newtheorem{proposition}{Proposition}
\newtheorem{theorem}{Theorem}
\newtheorem{lemma}{Lemma}

\DeclareMathOperator{\prox}{prox}

\DeclareMathOperator{\dist}{dist}

\author{Olivier Fercoq}
\title{Proximal gradient descent on the smoothed duality gap to solve saddle point problems}

\begin{document}

\maketitle

\begin{abstract}
In this paper, we minimize the self-centered smoothed gap, a recently introduced optimality measure, in order to solve convex-concave saddle point problems.
The self-centered smoothed gap can be computed as the sum of a convex, possibly nonsmooth function and a smooth weakly convex function. Although it is not convex, we propose an algorithm that minimizes this quantity, effectively reducing convex-concave saddle point problems to a minimization problem. Its worst case complexity
is comparable to the one of the restarted and averaged primal dual hybrid gradient method, and the algorithm enjoys linear convergence in favorable cases.
\end{abstract}


\section{Introduction}

In his celebrated paper ``Smooth minimization of non-smooth functions'' \cite{nesterov2005smooth}, Yurii Nesterov showed that a large class of Lipschitz continuous functions can be minimized much more efficiently than when using subgradient descent. The algorithm relies on a smoothing of the nonsmooth function at stake which is computable using the proximal operator of a simple elementary function. This work was later extended in \cite{tran2018smooth} to tackle non-Lipschitz functions, with a special focus on constraints. This was made possible by the introduction of the smoothed duality gap and the study of the relations between a small smoothed duality gap and feasibility and optimality gaps. Coordinate descent versions of these algorithms have also been proposed in~\cite{fercoq2017smooth,alacaoglu2017smooth}.

Yet, despite optimal worst case performance for the search of saddle points on convex-concave problems of the form $\min_x \max_y f(x) + \langle Ax, y\rangle - g^*(y)$, their practical performance on simple problems is disappointing. This comes from the fact that the convergence is governed by a pre-defined smoothing parameter sequence, which is set according to the worst case~\cite{tran2020adaptive}. 

In this work, we show how to overcome this drawback
by considering the Self-Centered Smoothed Duality Gap (SC-SDG). This quantity was introduced in~\cite{fercoq2023quadratic} as a computable proxy for the smoothed duality gap centered at a saddle point. It was then realized in~\cite{walwil2025primal} that it is in fact an optimality measure, and even a precise and versatile one, when compared to the Karush-Kuhn-Tucker error and the projected duality gap, which are used commonly to terminate primal-dual algorithms. SC-SDG has the very desirable property of taking value 0 at saddle points, whatever smoothing parameter we take. We propose to minimize it using the proximal gradient and accelerated proximal gradient methods. One difficulty 
is that even if the original problem is convex-concave, SC-SDG is a weakly convex function. Nevertheless, we can control the level of weak convexity and
by choosing the smoothing sequence properly, we can prove convergence to a global minimum, which has value 0.

Numerical experiments show that this worst case result does not reflect the simpler cases: our algorithms enjoy a performance comparable to PDHG \cite{chambolle2011first} and RAPDHG~\cite{fercoq2023quadratic,applegate2023faster} in a small linear program and a larger second order cone program on which we tested them.

\paragraph{Other related works} This work belongs to a long series of works on primal-dual algorithms for the resolution of convex-concave saddle point problems. We present here a very quick historical review that emphasizes the fundamental concepts rather than the technical subtleties.

%
%
%
%
%
%
%

One of the foundational methods is the gradient descent-ascent algorithm introduced by Arrow and Hurwicz~\cite{arrow1959studies}. This method applies gradient descent to the primal variables and gradient ascent to the dual variables. While elegant in its simplicity (it relies only on first-order information), its convergence is guaranteed only under strong assumptions such as strict convexity-concavity and small step sizes, limiting its practical applicability.

To address these convergence limitations, Korpelevich~\cite{korpelevich1976extragradient} proposed the extragradient method. This technique evaluates the gradient twice per iteration (first to estimate a better search direction and then to make the actual update), which improves the stability and convergence for a broader class of variational inequalities. Around the same time, Popov developed an alternative fix, modifying the update rule to use implicit information in a single projected step, offering another viable path to convergence.

Another line of work focused on dual methods. Uzawa~\cite{arrow1959studies} introduced gradient ascent on the dual problem, while Hestenes~\cite{hestenes1969multiplier} and Powell~\cite{powell1978algorithms} developed the Augmented Lagrangian Method (ALM), which can be seen as the proximal point method in the dual~\cite{rockafellar1976augmented}. These approaches are appealing because they operate over dual variables, transforming saddle point problems into unconstrained dual maximization problems. However, their main drawback lies in each iteration requiring the solution of an unconstrained minimization subproblem. Ensuring the accuracy of these inner solutions (so-called inexactness) can be challenging and computationally expensive.

Then, Gabay and Mercier~\cite{gabay1976dual} introduced the Alternating Direction Method of Multipliers (ADMM). ADMM decomposes problems into smaller subproblems that are easier to solve, by alternating updates between primal and dual variables while enforcing consistency via augmented Lagrangians. This splitting structure simplifies the inner minimization steps and enhances scalability.

More recently, Chambolle and Pock~\cite{chambolle2011first} proposed Primal-Dual Hybrid Gradient (PDHG) that fully splits the contributions of the objective terms $f$, $g$, and the linear operator $A$. This algorithm generalizes and clarifies earlier methods, with modern interpretations revealing its equivalence to ADMM after suitable variable changes, as shown in~\cite{oconnor2020equivalence}. The Chambolle-Pock method is notable for being simple, efficient, and widely applicable, especially in imaging and machine learning problems.

The promise of this work is to merge the benefits of the augmented Lagrangian method and primal-dual hybrid gradient method: an underlying minimization problem with a principled merit function together with simple iterations that can be performed exactly.

\section{Problem}

This goal of this paper is to present novel algorithms for the resolution of convex-concave saddle point problems with linear coupling:
\begin{align*}
\min_{x \in \mathbb R^n} \max_{y \in \mathbb R^m} f(x) + \langle Ax, y \rangle - g^*(y)
\end{align*}
where $f$ and $g$ are lower-semicontinuous convex functions whose proximal operator is easy to compute, also called proximable functions, and $A$ is a linear operator. Here, $g^*$ is the Fenchel conjugate of $g$.
We will also denote $\mathcal X = \mathbb R^n$, $\mathcal Y = \mathbb R^m$ and $\mathcal Z = \mathcal X \times \mathcal Y$.
We assume that the set of saddle points is not empty. We will denote it by $\mathcal Z_*$.

We shall manipulate the primal vector $x$ and the dual vector $y$ together within a concatenated primal-dual vector $z = (x,y) \in \mathcal Z = \mathcal X \times \mathcal Y$. We will then denote $F(z) = f(x) + g^*(y)$ and $M(z) = (-A^\top y, Ax)$.
The smoothed gap is the function defined, for a couple of positive parameters $(\beta_x, \beta_y)$ and a center $\dot z \in \mathcal Z$, by
\begin{multline*}
\bar G_{\beta_x, \beta_y}(z, \dot z) := \sup_{z'=(x', y') \in \mathcal Z} f(x) + \langle Ax, y' \rangle - g^*(y') - \frac{\beta_y}{2}\|y' - \dot y\|^2 \\ - f(x') - \langle Ax', y\rangle + g^*(y) - \frac{\beta_x}{2}\|x' - \dot x\|^2.
\end{multline*}
When $\beta_x = \beta_y = 0$, we recover the usual definition of the duality gap. For conciseness, we shall denote $\beta = (\beta_x, \beta_y) \in \mathbb R^2_+$.
In this paper, we shall consider the self-centered smoothed gap, that is the function $G_{\beta}(z)= \bar G_{\beta_x, \beta_y}(z, z)$:
\begin{align*}
G_{\beta}(z) :=  F(z) + \sup_{z'} \;\langle Mz, z'\rangle - F(z') - \frac{1}{2} \|z - z'\|^2_\beta \;.
\end{align*}
It will be convenient to denote $F_{\beta, M}^*(z) := \sup_{z'} \;\langle Mz, z'\rangle - F(z') - \frac{1}{2} \|z - z'\|^2_\beta$. We will see in the next section that $F^*_{\beta, M}$ is a smooth approximation of $F^* \circ M$.

\section{Basic properties}

The next result shows that the self-centered smoothed duality gap is an optimality measure, that is a nonnegative computable quantity that is 0 if and only if we are at a saddle point.
\begin{proposition}[Prop. 32 in \cite{fercoq2023quadratic}]
For all $\beta \geq 0$, $G_\beta(z) \geq 0$. \\ Moreover $z \in \mathcal Z_* \Leftrightarrow G_\beta(z) = 0$.
\end{proposition}

We now give some fundamental properties of $F^*_{\beta,M}$.
\begin{proposition}
The function $F^*_{\beta, M}$ is 1-weakly convex in the norm $\|\cdot\|_\beta$ and differentiable.
If we denote $\bar z_\beta(z) = \prox_{\beta^{-1}, F} (z + \frac 1 \beta M z)= \arg\min_{z'} F(z') - \langle Mz, z'\rangle + \frac{1}{2}\|z' - z\|_\beta^2$,
then the gradient of $F^*_{\beta, M}$ is given by $\nabla F^*_{\beta,M}(z) = -M \bar z_\beta(z) + \beta (\bar z_\beta(z)-z)$ and, in the norm $\|\cdot\|_{\gamma^{-1}}$ where $\gamma = (\gamma_x, \gamma_y)$, it is $L_{\beta,\gamma}$-Lipschitz for 
\begin{equation*}
L_{\beta,\gamma} =  \max(\beta_x \gamma_x + \frac{\gamma_x}{\beta_y}\|A\|^2, \beta_y\gamma_y + \frac{\gamma_y}{\beta_x}\|A\|^2)+\max(\beta_x\gamma_x, \beta_y\gamma_y)\;.
\end{equation*}
\end{proposition}
\begin{proof}
	We have
\begin{align*}
F^*_{\beta, M}(z) + \frac{1}{2}\|z\|_\beta^2 = c(z) := \sup_{z'} \;\langle Mz, z'\rangle - F(z') + \langle z, z'\rangle_\beta - \frac{1}{2} \|z'\|_\beta^2 \;.
\end{align*}
Hence $F^*_{\beta, M}$ can be written as the difference between a convex function and $\frac{1}{2}\|\cdot\|^2_\beta$, which proves the weak convexity.

The convex function $c$ defined above is the usual smoothing of nonsmooth functions \cite{nesterov2005smooth} applied to the function $F^*$ and the linear operator $M + \beta . I$, where the notation $\beta.I$ stands for $(\beta.I)z = (\beta_x x, \beta_y y)$.
Hence, it is differentiable, its gradient is $\nabla c(z) = (M^\top + \beta . I)\bar z_\beta(z)=(-M + \beta. I)\bar z_\beta(z)$ because $M$ is skew-symmetric. Moreover, $\nabla c$ is Lipschitz continuous with a Lipschitz constant using norms $\|\cdot\|_\gamma$ and $\|\cdot\|_{\gamma^{-1}}$ given by
\begin{align*}
\|\nabla c(z_1) &- \nabla c(z_2)\|_\gamma = \|(-M + \beta. I)\bar z_\beta(z_1)-(-M + \beta. I)\bar z_\beta(z_2)\|_\gamma \\
&\leq \|-M + \beta. I\|_{\gamma,\beta} \|\bar z_\beta(z_1) - \bar z_\beta(z_2)\|_\beta \\
& \leq \|-M + \beta. I\|_{\gamma,\beta} \|z_1+\beta^{-1}.Mz_1 - z_2-\beta^{-1}.Mz_2\|_\beta \\
& \leq \|-M + \beta. I\|_{\gamma,\beta} \|(\beta.I+M)(z_1-z_2)\|_{\beta^{-1}} \\
& \leq \|-M + \beta. I\|_{\gamma,\beta}\|M + \beta. I\|_{\beta^{-1},\gamma^{-1}}\|z_1 - z_2\|_{\gamma^{-1}}
\end{align*}
Now
\begin{align*}
\|-M+\beta.I\|_{\gamma,\beta} &= \sup_{z \neq 0} \frac{\|(-M+\beta.I)z\|_\gamma}{\|z\|_\beta} = \sup_{z' \neq 0} \frac{\|\gamma^{1/2}.(-M+\beta.I).\beta^{-1/2}.z'\|}{\|z\|} \\
&= \|\gamma^{1/2}.(-M+\beta.I).\beta^{-1/2}\| = \|\beta^{-1/2}.(M+\beta.I).\gamma^{1/2}\| = \|M + \beta. I\|_{\beta^{-1},\gamma^{-1}}
\end{align*}
so that the Lipschitz constant of $\nabla c$ is given by the largest eigenvalue of the matrix
\begin{equation*}
\begin{bmatrix}
\sqrt{\beta_x\gamma_x} I & \sqrt{\frac{\gamma_x}{\beta_y}}A^\top \\
-\sqrt{\frac{\gamma_y}{\beta_x}}A & \sqrt{\gamma_y\beta_y}I
\end{bmatrix} \times
\begin{bmatrix}
\sqrt{\beta_x\gamma_x} I & -\sqrt{\frac{\gamma_y}{\beta_x}}A^\top \\
\sqrt{\frac{\gamma_x}{\beta_y}}A & \sqrt{\gamma_y\beta_y}I
\end{bmatrix} 
= \begin{bmatrix}
\beta_x\gamma_x I + \frac{\gamma_x}{\beta_y}A^\top A & 0 \\
0 & \gamma_y\beta_y I + \frac{\gamma_y}{\beta_x}A A^\top
\end{bmatrix}
\end{equation*}
We obtain $L(\nabla c) = \max(\beta_x \gamma_x + \frac{\gamma_x}{\beta_y}\|A\|^2, \beta_y\gamma_y + \frac{\gamma_y}{\beta_x}\|A\|^2)$.
Hence, since $F^*_{\beta, M} = c(z) + \frac 1 2 \|z\|^2_\beta$, the Lipschitz constant of $\nabla F^*_{\beta, M}$ in the norm $\|\cdot\|_\gamma$ is given by $L_{\beta,\gamma} = \|M + \beta I\|^2_{\beta^{-1}, \gamma^{-1}} + \max(\beta_x \gamma_x, \beta_y\gamma_y)$.
\end{proof}

\section{Proximal gradient algorithm}

Our main idea is to run the proximal gradient algorithm on 
$G_\beta = F + F^*_{\beta, M}$. However, since this is not a convex function, the algorithm may converge to stationary points which are not minimizers.
To circumvent this issue, we consider a decreasing sequence of smoothing parameters $\beta_k$, so that the objective function gets closer and closer to a convex function.

\begin{algorithm}
\begin{align*}
&z_0 \in \mathcal Z \\
&\forall k \in \mathbb N: \\
& \qquad z_{k+1} = \prox_{\gamma_k F}(z_k - \gamma_k \nabla F_{\beta_k, M}^*(z_k))
\end{align*}
\caption{Proximal gradient descent on the self-centered smoothed gap with continuation} \label{alg:prox_grad_smoothed_gap}
\end{algorithm}

\begin{lemma}
	\label{lem:compare_gbetas}
Let $\beta$ and $\beta'$ be two couples of smoothing parameters. Then for all $z$,
\begin{equation*}
G_{\beta'}(z) \geq \big(2 - \max\big(\frac{\beta'_x}{\beta_{x}}, \frac{\beta'_y}{\beta_y}\big)\big) G_\beta(z) \;.
\end{equation*}
\end{lemma}
\begin{proof}
Because $G_\beta(z) = F(z) + \sup_{z'=(x',y')} \langle Mz, z'\rangle - F(z') - \frac{\beta_x}{2} \|x - x'\|^2 - \frac{\beta_y}{2} \|y - y'\|^2$, it is a convex function of $\beta$ and its gradient is given by $\frac{\partial G_\beta(z)}{\partial \beta_x} = -\frac{1}{2} \|x - \bar x_\beta(z)\|^2$ and $\frac{\partial G_\beta(z)}{\partial \beta_y} = -\frac{1}{2} \|y - \bar y_\beta(z)\|^2$. Thus we have
\begin{align*}
G_{\beta'}(z) &\geq G_{\beta}(z) - \frac{\beta_x' - \beta_x}{2}\|x - \bar x_\beta(z)\|^2 -\frac{\beta_y' - \beta_y}{2}\|y - \bar y_\beta(z)\|^2 
\end{align*}
Moreover,
\begin{align*}
& G_{\beta}(z) = \sup_{z'} F(z) - F(z') + \langle M z, z'\rangle -\frac{1}{2}\|z - z'\|^2_\beta \geq 0 + \frac{1}{2}\|z - \bar z_{\beta}(z)\|_\beta^2
\end{align*}
so 
\begin{align*}
G_{\beta'}(z_{k+1}) \geq \Big(2 - \max\big(\frac{\beta'_x}{\beta_{x}}, \frac{\beta'_y}{\beta_y}\big) \Big)G_{\beta}(z) \;.
\end{align*}
\end{proof}

\begin{theorem}
Let $p' > 0$ and $b = \frac{1}{ (3/2)^{2}-1} \approx 4.45$. If $\beta_{x,k} = \beta_{y,k}= \frac{\|M\|}{(k+b)^{1/2}} \sqrt{\frac{p'}{b+p'}}$ and $\gamma_{x,k} = \gamma_{y,k} = \frac{1}{L_{\beta_k, 1}} = \frac{\beta_{x,k}}{\|M\|^2 + 2 \beta_{x,k}^2}$ then the sequence $(z_k)$ generated by Algorithm~\ref{alg:prox_grad_smoothed_gap} satisfies
\begin{align*}
G_{\beta_{k+1}}(z_{k+1}) \leq \frac{e^1 (\|M\|^2+2\beta_0^2) \|z_0-z_*\|^2}{2 \beta_k \big(c_1(k+b+1)^{1-p'} -  c_2 \ln(k+b) + c_3\big)} \in O\Big(\frac{e^1 \|M\| \|z_0-z_*\|^2}{2\sqrt{p'}  k^{(1/2-p')}}\Big)
\end{align*}
where $c_1  \approx 1$, $c_2 \approx 0.49$ and $c_3 \approx -6.76$.
\end{theorem}
\begin{proof}
Since $\gamma_{x,k} = \gamma_{y,k}$, we will denote $\gamma_k = \gamma_{x,k}$ and similarly $\beta_k = \beta_{x,k}$. Let us consider $z_* \in \mathcal Z_*$. Note that for all $\beta$, $G_\beta(z^*) = 0$. We use successively the Taylor-Lagrange inequality, the 3-point inequality for the proximal operator, the weak convexity of $F^*_{\beta_k, M}$,
the condition $\gamma_{k} \leq \frac{1}{L_{\beta_k,1}}$ and $G_{\beta_k}(z^*)=0$.
\begin{align*}
G_{\beta_k}(z_{k+1}) &= F(z_{k+1}) + F^*_{\beta_k, M}(z_{k+1})\\
& \leq F(z_{k+1}) +  F^*_{\beta_k, M}(z_{k}) + \langle \nabla F^*_{\beta_k, M}(z_k), z_{k+1} - z_k\rangle + \frac{L_{\beta_k,1}}{2} \|z_{k+1} - z_k\|^2 \\
& \leq F(z^*) + F^*_{\beta_k, M}(z_{k}) + \langle \nabla F^*_{\beta_k, M}(z_k), z_{*} - z_k\rangle + \frac{1}{2\gamma_k}\|z_k - z_*\|^2 -\frac{1}{2\gamma_k}\|z_{k+1} - z_*\|^2 \\ & \qquad \qquad +\big(\frac{L_{\beta_k,1}}{2}-\frac{1}{2\gamma_k}\big) \|z_{k+1} - z_k\|^2 \\
& \leq F(z_*) + F^*_{\beta_k, M}(z_*)  + \big(\frac{\beta_k}{2} +\frac{1}{2\gamma_k}\big)\|z_k - z_*\|^2 -\frac{1}{2\gamma_k}\|z_{k+1} - z_*\|^2 \\
& \leq \big(\frac{\beta_k}{2} +\frac{1}{2\gamma_k}\big)\|z_k - z_*\|^2 -\frac{1}{2\gamma_k}\|z_{k+1} - z_*\|^2
\end{align*}
From $G_{\beta_k}(z_{k+1}) \leq F(z_{k+1}) +  F^*_{\beta_k, M}(z_{k}) + \langle \nabla F^*_{\beta_k, M}(z_k), z_{k+1} - z_k\rangle + \frac{L_{\beta_k, 1}}{2} \|z_{k+1} - z_k\|^2$, we can also get
\begin{align*}
G_{\beta_k}(z_{k+1}) \leq G_{\beta_k}(z_k) - \frac{1}{2 \gamma_k}\|z_{k+1} - z_k\|^2 \leq G_{\beta_k}(z_k)
\end{align*}
Hence, combining both inequality with factors $\delta_k$ and $1-\delta_k \in [0,1]$, we get
\begin{align*}
G_{\beta_k}(z_{k+1}) \leq (1-\delta_k) G_{\beta_k}(z_k)+\big(\frac{\delta_k\beta_k}{2} +\frac{1}{2\gamma_k}\big)\|z_k - z_*\|^2 -\frac{1}{2\gamma_k}\|z_{k+1} - z_*\|^2
\end{align*}
By Lemma \ref{lem:compare_gbetas},
$G_{\beta_k}(z_{k+1}) \geq \big(2 - \frac{\beta_k}{\beta_{k+1}} \big)G_{\beta_{k+1}}(z_{k+1}).$

We obtain
\begin{align*}
\big(2 - \frac{\beta_k}{\beta_{k+1}} \big)G_{\beta_{k+1}}(z_{k+1}) \leq (1-\delta_k) G_{\beta_k}(z_k)+\big(\frac{\delta_k\beta_k}{2} +\frac{1}{2\gamma_k}\big)\|z_k - z_*\|^2 -\frac{1}{2\gamma_k}\|z_{k+1} - z_*\|^2
\end{align*}

We multiply both sides by $\gamma_k$ and
we iterate for $k \in \{0, \ldots, K-1\}$:
\begin{align}
\frac 12 \|z_{k+1} - z^*\|^2 &\leq \frac 12 \big(\prod_{l=0}^k (1+\gamma_l \beta_l\delta_l)\big) \|z_0 - z_*\|^2 \notag\\
&\qquad\qquad- \sum_{l=0}^k \big(\prod_{j=l+1}^{k} (1+\gamma_j \beta_j \delta_j)\big) \gamma_l \big((2-\frac{\beta_l}{\beta_{l+1}}) G_{\beta_{l+1}}(z_{l+1}) - (1-\delta_l) G_{\beta_{l}}(z_{l}) \big)\notag\\
& = \frac 12 \big(\prod_{l=0}^k (1+\gamma_l \beta_l\delta_l)\big) \|z_0 - z_*\|^2 + \big(\prod_{j=1}^{k} (1+\gamma_j \beta_j \delta_j)\big)(1-\delta_0) G_{\beta_{0}}(z_{0})\notag\\
& \qquad\qquad - \sum_{l=1}^k \big(\prod_{j=l}^{k} (1+\gamma_j \beta_j \delta_j)\big) \big(\gamma_{l-1} (2-\frac{\beta_{l-1}}{\beta_{l}}) G_{\beta_{l}}(z_{l}) - \frac{\gamma_l (1-\delta_l)}{1+\gamma_l\beta_l\delta_l} G_{\beta_{l}}(z_{l}) \big) \notag \\
& \qquad\qquad- \gamma_k(2-\frac{\beta_{k}}{\beta_{k+1}}) G_{\beta_{k+1}}(z_{k+1})
 \;.\label{sum_smoothed_gap_in_prox_grad}
\end{align}
Because $M$ is skew-symmetric, $\|M+\beta_k I\|^2= \|M\|^2 + \beta_k^2$. Thus, $\gamma_k = \frac{1}{\frac{1}{\beta_k}\|M + \beta_kI\|^2 + \beta_k} = \frac{\beta_k}{\|M\|^2 + 2\beta_k^2} \leq \frac{\beta_k}{\|M\|^2}$.
Recall that for $q > 1$ and $b>0$, $\sum_{k=0}^{+\infty} \frac{1}{(k + b)^q} \leq b^{-q} + \sum_{k=1}^{+\infty} \int_{k+b-1}^{k+b} \frac{1}{x^q} dx \leq b^{-q} + \frac{b^{1-q}}{q-1}$.
Choosing $\beta_k = \frac{\beta_0}{(k/b+1)^{p}} = \frac{\beta_0 b^p}{(k+b)^{p}}$ and $\delta_k = \frac{1}{(k/b+1)^{p'}}$ yields, as soon as $2p+p'>1$,
\begin{align*}
\sum_{k=0}^{+\infty}\gamma_k \beta_k \delta_k\leq \sum_{k=0}^{+\infty} \frac{\beta_k^2\delta_k}{\|M\|^2} = \frac{\beta_0^2 b^{2p+p'}}{\|M\|^2}\sum_{k=0}^{+\infty} \frac{1}{(k+b)^{2p+p'}} \leq \frac{\beta_0^2}{\|M\|^2} + \frac{\beta_0^2}{\|M\|^2}\frac{b}{2p + p' - 1} := \ell\;.
\end{align*}
Hence, $\ln(\prod_{l=0}^{k} (1 + \beta_l \gamma_l\delta_l) ) \leq \sum_{l=0}^k \beta_l \gamma_l\delta_l \leq \ell$, so $\prod_{l=0}^{k} (1 + \beta_l \gamma_l\delta_l) \leq \exp(\ell)$.
Note also that our choice $\delta_k = \frac{1}{(k/b+1)^{p'}}$ implies that $\delta_0 = 1$.
Because $p \leq 1$, by concavity of the function $(x \mapsto (1+x)^p)$, we have
\begin{align*}
-\gamma_{l-1} (2-\frac{\beta_{l-1}}{\beta_{l}}) + \frac{\gamma_l (1-\delta_l)}{1+\gamma_l\beta_l\delta_l} &\leq  \gamma_{l-1}(-2+\frac{\beta_{l-1}}{\beta_l}+1-\delta_l) \\
&= \gamma_{l-1}\Big(-1+\frac{(l+b)^p}{(l+b-1)^{p}}-\frac{1}{(l/b+1)^{p'}}\Big)\\
&= \gamma_{l-1}\Big(-1+\big(1 + \frac{1}{l+b-1}\big)^p-\frac{1}{(l/b+1)^{p'}}\Big)\\
&\leq \gamma_{l-1}\Big(\frac{p}{l+b-1}-\frac{1}{(l/b+1)^{p'}}\Big)
\end{align*}
This quantity is negative as soon as 
$p' \leq 1$ and $1 \geq \frac{p}{b-1}$.
Combining all these results in~\eqref{sum_smoothed_gap_in_prox_grad}, we get
\begin{align*}
\sum_{l=1}^k \gamma_{l-1} \Big(\frac{1}{(l/b+1)^{p'}}-\frac{p}{l+b-1}\Big)G_{\beta_l}(z_{l}) + \gamma_k(2-\frac{\beta_{k}}{\beta_{k+1}}) G_{\beta_{k+1}}(z_{k+1}) \leq \frac{\exp(\ell)}{2}\|z_0 - z_*\|^2 \;.
\end{align*}


We shall gather all the terms of the left hand side by using
\begin{align*}
G_{\beta_l}(z_{l})& \geq \big(2 - \frac{\beta_l}{\beta_{l+1}} \big) G_{\beta_{l+1}}(z_{l+1}) \geq \prod_{j=l}^{K} \big(2 - \frac{\beta_{j}}{\beta_{j+1}} \big) G_{\beta_{K+1}}(z_{K+1})
\end{align*}

Let us denote $\phi(x) =  \ln \big(2-(1+x)^p\big)$, so that for $\beta_j = \frac{\beta_0}{(j/b+1)^p}$, we have the equality
$\ln\Big(\prod_{j=l}^{K} \big(2 - \frac{\beta_{j}}{\beta_{j+1}}\big)\Big) = \sum_{j=l}^K \phi(1/(j+b))$.

For all $x < 2^{1/p} - 1$,
\begin{align*}
&\phi'(x) = \frac{-p(1+x)^{p-1}}{2 - (1+x)^p} \\
&\phi''(x) = \frac{p(1-p)(1+x)^{p-2}}{2 - (1+x)^p} - \frac{p^2(1+x)^{2p-2}}{(2-(1+x)^p)^2} \\
&\phi'''(x) = \frac{p(1-p)(p-2)(1+x)^{p-3}}{2 - (1+x)^p} - \frac{p^2(2p-2)(1+x)^{2p-3}}{(2-(1+x)^p)^2} + 2 \frac{p^3(1+x)^{3p-3}}{(2-(1+x)^p)^3}
\end{align*}
When $p = \frac 12$ and $\phi'(0) = -p$, $\phi''(0) = 0$. Moreover, for $0 \leq x \leq (\frac{3}{2})^{1/p}-1$, $1\leq (1+x)^p \leq \frac{3}{2}$ and $|\phi'''(x)| \leq \frac{3}{8}\frac{1}{2-3/2} + \frac{1}{4}\frac{1}{(2-3/2)^2} + \frac{1}{4}\frac{1}{(2-3/2)^3} = \frac{3}{4} + 1 + 2 \leq 4$. Note that as soon as $b \geq \frac{1}{ (3/2)^{1/p}-1} \approx 4.45$, we have $\frac{1}{j+b} \leq (3/2)^{1/p}-1$ for all $j\in \mathbb N$. This implies that for all $x \in [0, (3/2)^{1/2}-1]$,
\begin{align*}
&\phi(x) \geq -\frac{1}{2} x - \frac{4}{6} x^3\\
& \sum_{j=l}^K \phi(1/(j+b)) \geq -\frac 12 \sum_{j=l}^K \frac{1}{j+b} - \frac{2}{3}\sum_{j=l}^K \frac{1}{(j+b)^3} \\
& \phantom{\sum_{j=l}^K \phi(1/(j+b))}\geq 
-\frac{1}{2}\Big(\ln(K+b) - \ln(l+b-1)\Big) - \frac{2}{3} \frac{(l+b-1)^{-2}}{2} \\
& \prod_{j=l}^K (2 - \frac{\beta_j}{\beta_{j+1}}) \geq \Big(\frac{l+b-1}{K+b}\Big)^{1/2} \times \exp(\frac{-1}{3(l+b-1)^2})\geq \exp(\frac{-1}{3b^2}) \frac{\beta_K}{\beta_{l-1}}
\end{align*}
where the last inequality is true as soon as $l \geq 1$. Thus, using also $\gamma_k \geq \frac{\beta_k}{\|M\|^2 + 2 \beta_0^2}$,
\begin{align*}
\frac{\exp(\ell)}{2}&\|z_0 - z_*\|^2 \geq \sum_{l=1}^k \gamma_{l-1} \Big(\frac{1}{(l/b+1)^{p'}}-\frac{p}{l+b-1}\Big)G_{\beta_l}(z_{l}) + \gamma_k(2-\frac{\beta_{k}}{\beta_{k+1}}) G_{\beta_{k+1}}(z_{k+1}) \\
&\geq \Big(\sum_{l=1}^k \gamma_{l-1} \exp(\frac{-1}{3b^2})\big(\frac{1}{(l/b+1)^{p'}}-\frac{p}{l+b-1}\big)\frac{\beta_k}{\beta_{l-1}} + \gamma_k(2-\frac{\beta_{k}}{\beta_{k+1}})\Big) G_{\beta_{k+1}}(z_{k+1}) \\
& \geq \Big(\frac{\exp(\frac{-1}{3b^2})}{\|M\|^2 + 2\beta_0^2}\sum_{l=1}^k \big(\frac{1}{(l/b+1)^{p'}}-\frac{p}{l+b-1}\big) \beta_k + \beta_k\frac{2-\beta_0/\beta_1}{\|M\|^2+2\beta_0^2}\Big) G_{\beta_{k+1}}(z_{k+1})
\end{align*}
Since $p'<1$,
\begin{align*}
\sum_{l=1}^k \frac{1}{(l/b+1)^{p'}} &= b^{p'} \sum_{l=1}^k \frac{1}{(l+b)^{p'}}\geq b^{p'} \sum_{l=1}^k \int_{l+b}^{l+b+1} \frac{1}{x^{p'}}dx = b^{p'} \int_{b+1}^{k+b+1}\frac{1}{x^{p'}}dx \\
&\geq \frac{b^{p'}}{1-p'} \Big((k+b+1)^{1-p'} - (b+1)^{1-p'} \Big)
\end{align*}
We also have
\begin{align*}
\sum_{l=1}^k \frac{1}{l+b-1} \leq \frac{1}{b-1} + \ln(k+b) - \ln(b-1)
\end{align*}
so that
\begin{align*}
\exp(\frac{-1}{3b^2})&\sum_{l=1}^k \big(\frac{1}{(l/b+1)^{p'}}-\frac{p}{l+b-1}\big) + (2-\beta_0/\beta_1) \\
&\geq \exp(\frac{-1}{3b^2}) \frac{b^{p'}}{1-p'} \Big((k+b+1)^{1-p'} - (b+1)^{1-p'} \Big) \\
& \qquad - \frac 12 \exp(\frac{-1}{3b^2}) \big(\frac{1}{b-1} + \ln(k+b) - \ln(b-1)\big)+ 2 - (b+1)^{1/2} \\
&=\exp(\frac{-1}{3b^2}) \frac{b^{p'}}{1-p'} (k+b+1)^{1-p'} - \frac 12 \exp(\frac{-1}{3b^2}) \ln(k+b) \\
&\qquad - \exp(\frac{-1}{3b^2}) \frac{b^{p'}}{1-p'}(b+1)^{1-p'}  - \frac 12 \exp(\frac{-1}{3b^2}) \big(\frac{1}{b-1}- \ln(b-1)\big) + 2 - (b+1)^{1/2} \\
& := c_1 (k+b+1)^{1-p'} - c_2 \ln(k+b) + c_3
\end{align*}
Finally,
\begin{align*}
G_{\beta_{k+1}}(z_{k+1}) \leq \frac{\exp(\ell) (\|M\|^2+2\beta_0^2) \|z_0-z_*\|^2}{2 \beta_k \big(c_1(k+b+1)^{1-p'} -  c_2 \ln(k+b) + c_3\big)}
\end{align*}
We shall choose $p=0.5$, $b=\frac{1}{ (3/2)^{1/p}-1}$ and $p'$ positive but small.

 Then $2 \beta_k \big(c_1(k+b+1)^{1-p'} -  c_2 \ln(k+b) + c_3\big) \sim_{k\to+\infty} 2 \beta_0 c_1 k^{1/2-p'}$ where $c_1 \approx 1$.

Since $\ell = \frac{\beta_0^2}{\|M\|^2}(1 + \frac{b}{p'})$, we should compensate for the smallness of $p'$ by smartly choosing $\beta_0$ to avoid an overwhelming constant when applying the exponential. Taking $\beta_0 = \|M\|\sqrt{\frac{p'}{b+p'}}$ leads to $\exp(\ell) = e^1$, which is reasonable.

With these considerations, we get the simplified rate estimate
\begin{align*}
G_{\beta_{k}}(z_{k}) \in O\Big(\frac{e^1 \|M\| \|z_0-z_*\|^2}{2\sqrt{p'}  k^{1/2-p'}}\Big) \;.
\end{align*}
\end{proof}

\section{Accelerated proximal gradient descent}

In this section, we propose to adapt the accelerated proximal gradient descent algorithm to the minimization of the self-centered smoothed gap. The algorithm is described in Algorithm~\ref{alg:acc_norestart}. 
In addition to the acceleration, we are going to write the proof using weighted norms to allow for different primal and dual step sizes.

\begin{algorithm}
	\begin{align*}
	\bar z_0 = \;&z_0 \in \mathcal Z \\
	\forall k \in &\;\mathbb N:\\
	& \hat z_{k} = (1-\theta_k) z_k + \theta_k\bar z_k \\
	& \bar z_{k+1} = \prox_{\frac{\gamma_k}{\theta_k}, F}\Big(\bar z_k - \frac{\gamma_k}{\theta_k}.\nabla F^*_{\beta_k, M}(\hat z_k) \Big)\\
	& z_{k+1} = (1-\theta_k)z_k + \theta_k \bar z_{k+1}
	\end{align*}
	\caption{Accelerated proximal gradient descent for the smoothed gap} \label{alg:acc_norestart}
\end{algorithm}

\begin{theorem}
	\label{thm:acc}
We consider the iterates of Algorithm~\ref{alg:acc_norestart}. Suppose that
$\theta_k = \frac{t}{k+t}$, $\beta_k = \frac{\beta_0 b}{k + b}$, 
$\gamma_{x,k} = \frac{\beta_{y,k}}{2\beta_{x,k}\beta_{y,k}+\|A\|^2}$ and $\gamma_{y,k} = \frac{\beta_{x,k}}{2\beta_{y,k}\beta_{x,k}+\|A\|^2}$.
where $t$, $b$ and $\beta_0$ satisfy
\begin{align*}
&b \geq t \geq 2 \\
&\bar c = \frac{\beta_{x,0} \beta_{y,0} b^2}{t \|M\|^2} < 1 
\end{align*}
Then for all $K \geq 1$,
\begin{align*}
G_{\beta_{K}}(z_{K}) \leq \frac{b}{2}\frac{1}{(K+b-2)^{1-\bar c}}  \| z_0 - z^*\|^2_{\gamma_0^{-1}+\beta_0} \in O(K^{\bar c - 1})
\end{align*}
\end{theorem}
\begin{proof}
The proof starts like \cite{tseng2008accelerated} but then includes the effect of weak convexity.
The second line uses convexity of $F$ and smoothness of $F^*_{\beta_k, M}$, the third line uses $z_{k+1} = \hat z_k + \theta_k (\bar z_{k+1} - \bar z_k)$, the forth line is due to the fact that $z_{k+1}$ is the result of a proximal operator, the fifth line uses $\theta_k \bar z_k = \hat z_k - (1-\theta_k)z_k = \theta_k \hat z_k + (1-\theta_k)(\hat z_k - z_k)$, the sixth line uses twice the $\beta_k$-weak convexity of $F^*_{\beta_k, M}$.
\begin{align}
G_{\beta_k}&(z_{k+1}) = F(z_{k+1}) + F^*_{\beta_k, M}(z_{k+1}) \notag\\
&\leq \theta_k F(\bar z_{k+1}) + (1-\theta_k) F(z_k) + F^*_{\beta_k, M}(\hat z_k) + \langle \nabla F^*_{\beta_k, M}(\hat z_k), z_{k+1} - \hat z_k\rangle + \frac{L_{\beta_k, \gamma_k}}{2}\|z_{k+1} - \hat z_k\|^2_{\gamma_k^{-1}} \notag\\
& = (1-\theta_k) F(z_k) + F^*_{\beta_k, M}(\hat z_k) + \theta_k \Big( F(\bar z_{k+1}) + \langle \nabla F^*_{\beta_k, M}(\hat z_k), \bar z_{k+1} - \bar z_k\rangle \Big)\notag \\
& \qquad\qquad+ \frac{L_{\beta_k,\gamma_k}\theta_k^2}{2}\|\bar z_{k+1} - \bar z_k\|^2_{\gamma_k^{-1}}\notag \\
& \leq  (1-\theta_k) F(z_k) + F^*_{\beta_k, M}(\hat z_k) + \theta_k \Big( F(z^*) + \langle \nabla F^*_{\beta_k, M}(\hat z_k), z^* - \bar z_k\rangle + \frac{\theta_k}{2}\|z^* - \bar z_k\|^2_{\gamma_k^{-1}}  \notag\\
& \qquad\qquad- \frac{\theta_k}{2}\|z^* - \bar z_{k+1}\|^2_{\gamma_k^{-1}} \Big)+ \big(\frac{L_{\beta_k,\gamma_k}\theta_k^2}{2} - \frac{\theta_k^2}{2} \big)\|\bar z_{k+1} - \bar z_k\|^2_{\gamma_k^{-1}} \notag\\
& =  (1-\theta_k) F(z_k) + \theta_k F(z^*) + \theta_k F^*_{\beta_k, M}(\hat z_k) + \theta_k \langle \nabla F^*_{\beta_k, M}(\hat z_k), z^* - \hat z_k\rangle \notag\\
&\qquad\qquad +(1-\theta_k)F^*_{\beta_k, M}(\hat z_k)+(1-\theta_k)\langle \nabla F^*_{\beta_k, M}(\hat z_k), z_k - \hat z_k\rangle  + \frac{\theta_k^2}{2}\|z^* - \bar z_k\|^2_{\gamma_k^{-1}}\notag\\
& \qquad\qquad - \frac{\theta_k^2}{2}\|z^* - \bar z_{k+1}\|^2_{\gamma_k^{-1}} + (L_{\beta_k,\gamma_k}-1)\frac{\theta_k^2}{2} \|\bar z_{k+1} - \bar z_k\|^2_{\gamma_k^{-1}}\notag \\
& \leq (1-\theta_k) F(z_k) + \theta_k F(z^*) + \theta_k F^*_{\beta_k, M}(z^*) + \frac{\theta_k}{2}\|z^* - \hat z_k\|^2_{\beta_k}\notag \\
&\qquad\qquad +(1-\theta_k)F^*_{\beta_k, M}(z_k)+\frac{1-\theta_k}{2}\|z_k - \hat z_k\|^2_{\beta_k}  + \frac{\theta_k^2}{2}\|z^* - \bar z_k\|^2_{\gamma_k^{-1}} - \frac{\theta_k^2}{2}\|z^* - \bar z_{k+1}\|^2_{\gamma_k^{-1}} \notag\\
& \qquad\qquad+ (L_{\beta_k, \gamma_k}-1)\frac{\theta_k^2}{2} \|\bar z_{k+1} - \bar z_k\|^2_{\gamma_k^{-1}}
\notag
\end{align}
We can work a bit the additional terms coming from weak convexity:
\begin{align*}
&\frac{ \theta_k}{2}\|z^* - \hat z_k\|^2_{\beta_k} + \frac{1-\theta_k}{2}\|z_k - \hat z_k\|^2_{\beta_k} \\
& = \frac{\theta_k(1-\theta_k)}{2}\|z^* - z_k\|^2_{\beta_k } + \frac{\theta_k^2}{2}\|z^* - \bar z_k\|^2_{\beta_k } - \frac{ \theta_k^2(1-\theta_k)}{2}\|\bar z_k - z_k\|^2_{\beta_k} + (1-\theta_k)\frac{\theta_k^2}{2}\|z_k - \bar z_k\|^2_{\beta_k} \\
& =  \frac{ \theta_k(1-\theta_k)}{2}\|z^* - z_k\|^2_{\beta_k} + \frac{\theta_k^2}{2}\|z^* - \bar z_k\|^2_{\beta_k }
\end{align*}
Moreover, 
\begin{align*}
&\frac{\bar a \beta_k}{2} \|z_{k+1} - z^*\|^2 \leq \frac{\bar a \beta_k}{2} (1-\theta_k) \|z_k - z^*\|^2 + \frac{\bar a \beta_k}{2}\theta_k\|\bar z_{k+1} - z^*\|^2 \\
&G_{\beta_k}(z_{k+1}) \geq \Big(2 - \max\big(\frac{\beta_{x,k}}{\beta_{x,k+1}}, \frac{\beta_{y,k}}{\beta_{y, k+1}}\big)\Big)G_{\beta_{k+1}}(z_{k+1})
\end{align*}
Combining the four formulas and using the fact that $G_{\beta_k}(z^*) = 0$ and $\beta_{x,k}/\beta_{x, k+1} = \beta_{y,k}/\beta_{y, k+1}$, we obtain the following Lyapunov-like inequality
\begin{align}
\big(2 - \frac{\beta_{x,k}}{\beta_{x,k+1}}\big)&G_{\beta_{k+1}}(z_{k+1}) +\frac{1}{2}\|z_{k+1} - z^*\|^2_{\beta_k} + \frac{\theta_k^2}{2}\|\bar z_{k+1} - z^*\|^2_{\gamma_k^{-1}} - \frac{\theta_k}{2}\|\bar z_{k+1} - z^*\|^2_{\beta_k} \notag\\
& \leq (1-\theta_k) G_{\beta_k}(z_k) + \frac{(1-\theta_k)(1+\theta_k)}{2}  \|z_k - z^*\|^2_{\beta_k} + \frac{\theta_k^2}{2}\|\bar z_k - z^*\|^2_{\gamma_k^{-1}}\notag\\
& \qquad \qquad + \frac{\theta_k^2}{2} \|\bar z_k - z^*\|^2_{\beta_k} + (L_{\beta_k, \gamma_k}-1)\frac{\theta_k^2}{2} \|\bar z_{k+1} - \bar z_k\|^2_{\gamma_k^{-1}}
\label{first_bound_on_Gbeta_for_acc}
\end{align}

We would like to choose $\gamma_k$ in such a way that $L_{\beta_k, \gamma_k} \leq 1$. 
This is possible with $\gamma_k$ such that
\begin{align*}
\gamma_{x,k} = \frac{\beta_{y,k}}{2\beta_{x,k}\beta_{y,k}+\|A\|^2}\quad \text{and} \quad \gamma_{y,k} = \frac{\beta_{x,k}}{2\beta_{y,k}\beta_{x,k}+\|A\|^2}\;.
\end{align*}

Given $\beta_{x,0} >0, \beta_{y,0}>0, b>0, t>0$, we choose 
\begin{align*}
& \beta_k = \frac{\beta_0 b}{k + b} \\
& \theta_k = \frac{t}{k + t}
\end{align*}
This choice of parameters ensures:
\begin{align*}
& 2 - \frac{\beta_{x,k}}{\beta_{x,k+1}} = 2 - \frac{k+b+1}{k+b} = \frac{k+b-1}{k+b} = \frac{\beta_{x,k}}{\beta_{x,k-1}} \\
&1-\theta_k = \frac{k}{k+t}
\end{align*}

\begin{align*}
&\frac{(1-\theta_k) \beta_{x,k-2}}{\beta_{x,k-1}} = \frac{k(k+b-1)}{(k+t)(k+b-2)}  \\
& \frac{\beta_{x,k}(1-\theta_k^2)}{\beta_{x,k-1}}\leq \frac{\beta_{x,k}}{\beta_{x,k-1}}=\frac{k+b-1}{k+b}
\end{align*}
\begin{align*}
\frac{\frac{\theta_k^2}{\gamma_{x,k}} + \beta_{x,k} \theta_k^2}{\frac{\theta_{k-1}^2}{\gamma_{x,k-1}}-\beta_{x,k-1}\theta_{k-1}} &= \frac{\theta_k^2 \|A\|^2/\beta_{y,k} + 3 \theta_k^2 \beta_{x,k}  }{\theta_{k-1}^2 \|A\|^2/ \beta_{y,k-1} + 2  \beta_{x,k-1} \theta_{k-1}^2- \beta_{x,k-1} \theta_{k-1}} \\
&= \frac{\theta_{k}^2\beta_{y,k-1}}{\theta_{k-1}^2\beta_{y,k}}\frac{\|A\|^2 + 3 \beta_{x,k}\beta_{y_k}}{\|A\|^2 + 2 \beta_{x,k-1}\beta_{y,k-1}(1-1/\theta_{k-1})}
\end{align*}
\begin{align*}
&\frac{\|A\|^2 + 3 \beta_{x,k}\beta_{y,k}}{\|A\|^2 + 2 \beta_{x,k-1}\beta_{y_{k-1}}(1-1/\theta_{k-1})} \leq 1 + \frac{\bar c}{k + c} \\
&\Leftrightarrow \|A\|^2 + 3 \beta_{x,k}\beta_{y,k} \leq \|A\|^2 + \frac{\bar c}{k + c} \|A\|^2 + (2 \beta_{x,k-1}\beta_{y,k-1}(1-1/\theta_{k-1})) \frac{k+c+\bar c}{k + c}\\
& \Leftrightarrow \frac{3\beta_{x,0}\beta_{y,0} b^2}{(k+b-1)^2} + \frac{\beta_{x,0}\beta_{y,0}b^2(k+t-1)(k+c+\bar c)}{t(k+b-1)^2(k+c)}\leq  \frac{\bar c}{k+c} \|A\|^2+ \frac{2\beta_{x,0}\beta_{y,0}b^2 (k+c+\bar c)}{(k+b-1)^2(k+c)} \\
& \Leftrightarrow 3\beta_{x,0}\beta_{y,0} b^2(k+c) +  \frac{\beta_{x,0}\beta_{y,0} b^2}{t} (k+t-1)(k+c+\bar c) \leq \bar c\|A\|^2 (k+b-1)^2 + 2 \beta_{x,0}\beta_{y,0}b^2(k+c+\bar c) 
\end{align*}
Taking
$\bar c = \frac{\beta_{x,0}\beta_{y,0} b^2}{t \|A\|^2}$ allows us to remove the quadratic term. There remains
\begin{align*}
&3\beta_{x,0}\beta_{y,0} b^2(k+c) +  \bar c \|A\|^2 (k(t-1+c+\bar c)+(t-1)(c+\bar c)) &\\
&\qquad\qquad\leq \bar c\|A\|^2 (2k(b-1)+(b-1)^2) + 2 \beta_{x,0}\beta_{y,0}b^2(k+c+\bar c) \\
&\Leftrightarrow k \Big(t\bar c \|A\|^2 + \bar c \|A\|^2 (t+c+\bar c-2b+1)\Big) + \Big(t\bar c \|A\|^2 (c-2\bar c)+ \bar c\|A\|^2((t-1)(c+\bar c)-(b-1)^2)\Big)\leq 0 \\
& \Leftrightarrow 2t+c+\bar c -2b - 1 \leq 0 \quad \text{and} \quad c-2\bar c + (t-1)(c+\bar c)-(b-1)^2 \leq 0 \\ 
& \Leftarrow c = \min(2b+1-2t-\bar c, \frac{(b-1)^2-(t-3)\bar c}{t})
\end{align*}
We can do exactly the same analysis when replacing $\beta_{x,k}$ by $\beta_{y,k}$ and $\gamma_{x,k}$ by $\gamma_{y,k}$.
We shall assume that the parameters are chosen in such a way that $\bar c = \frac{\beta_{x,0}\beta_{y,0} b^2}{t \|A\|^2} < 1$. Note moreover that when $b \geq t \geq 2$, then $c \geq 1$.

Then, denoting $\|\beta\|_\infty = \max(|\beta_x|, |\beta_y|)$ for $\beta \in \mathbb R^2$,
\begin{align*}
\prod_{k=1}^K &\Big\|\frac{\|A\|^2 + 3 \beta_k^2}{\|A\|^2 + 2 \beta_{k-1}^2-\beta_{k-1}^2/\theta_{k-1}}\Big\|_\infty \leq \prod_{k=1}^K (1 + \frac{\bar c}{k + c}) \leq \exp(\sum_{k=1}^K \ln(1+\bar c/(k+c)))  \\
&\leq \exp(\sum_{k=1}^K \frac{\bar c}{k+c}) 
\leq \exp(\bar c \ln(K+c)-\bar c\ln(c)) = \frac{(K+c)^{\bar c}}{c^{\bar c}} = (K/c+1)^{\bar c}
\end{align*}

We now go back to \eqref{first_bound_on_Gbeta_for_acc}. It can be rewritten
\begin{align*}
\frac{\beta_{x,k}}{\beta_{x,k-1}}G_{\beta_{k+1}}(z_{k+1}) & +\frac{1}{2}\|z_{k+1} - z^*\|^2_{\beta_k} + \frac 12\|\bar z_{k+1} - z^*\|^2_{\theta_k^2\gamma_k^{-1} -\beta_k\theta_k} \\
&\leq (1-\theta_k)\frac{\beta_{x,k-2}}{\beta_{x,k-1}}\frac{\beta_{x,k-1}}{\beta_{x,k-2}} G_{\beta_k}(z_k) + \frac{\beta_{x,k}}{\beta_{x,k-1}} (1-\theta_k^2) \frac{1}{2}\|z_k - z^*\|^2_{\beta_{k-1}}  \\
& \qquad\qquad+ \Big\|\frac{\frac{\theta_k^2}{2\gamma_k} +\frac{\beta_k\theta_k^2}{2}}{\frac{\theta_{k-1}^2}{2\gamma_{k-1}}-\frac{\beta_{k-1}\theta_{k-1}}{2}}\Big\|_\infty\frac 12 \|\bar z_k - z^*\|^2_{\theta_{k-1}^2 \gamma_{k-1}^{-1} - \beta_{k-1}\theta_{k-1}}
\end{align*}
Denote 
\begin{equation}
\mathcal L_k = \frac{\beta_{x,k-1}}{\beta_{x,k-2}}G_{\beta_{k}}(z_{k}) +\frac{1}{2}\|z_{k} - z^*\|^2_{\beta_{k-1}} + \frac 12 \|\bar z_k - z^*\|^2_{\theta_{k-1}^2 \gamma_{k-1}^{-1} - \beta_{k-1}\theta_{k-1}}
\label{eq:def_lyap_acc}
\end{equation}
 and
\begin{equation*}
\rho_k = \max\Big((1-\theta_k)\frac{\beta_{x,k-2}}{\beta_{x,k-1}}, \frac{\beta_{x,k}}{\beta_{x,k-1}}, \frac{\theta_k^2 \beta_{x,k-1}}{\theta_{k-1}^2 \beta_{x,k}} \frac{k+c+\bar c}{k+c} \Big)
\end{equation*}
then we have
\begin{equation*}
\mathcal L_{k+1} \leq \rho_k \mathcal L_k \;.
\end{equation*}
\begin{align*}
\rho_k &= \max\Big(\frac{k(k+b-1)}{(k+t)(k+b-2)}, \frac{k+b-1}{k+b}, \frac{(k+t-1)^2(k+b)(k+c+\bar c)}{(k+t)^2(k+b-1)(k+c)}\Big) \\
&\leq 
\max\Big(\frac{k(k+b-1)}{(k+t)(k+b-2)}, \frac{k+b-1}{k+b}, \frac{(k+t-1)^2(k+b)}{(k+t)^2(k+b-1)}\Big)\frac{k+c+\bar c}{k+c}
\end{align*}
In order to simplify the expression of $\rho_k$, we shall compare the three terms in the maximum.
We have:
\begin{align*}
&(t\geq 2, b\geq 2) \Rightarrow \frac{k(k+b-1)}{(k+t)(k+b-2)} \leq  \frac{k+b-1}{k+b} \\
& b \geq t \Leftrightarrow \frac{k+b-1}{k+b} \geq \frac{(k+t-1)^2(k+b)}{(k+t)^2(k+b-1)}
\end{align*}
Thus, in the regime $b \geq t \geq 2$, we get
$\rho_k \leq \frac{k+b-1}{k+b}\frac{k+c+\bar c}{k+c}$ and 
\begin{align*}
\prod_{k=1}^{K} \rho_k \leq \frac{b}{K+b} (K/c+1)^{\bar c} \;.
\end{align*}
We do a special treatment for $\mathcal L_1$, using the fact that $\theta_0 = 1$, as follows
\begin{align*}
\mathcal L_1 &= (2-\frac{\beta_{x,0}}{\beta_{x,1}})G_{\beta_1}(z_1) + \frac{1}{2}\|z_1-z^*\|^2_{\beta_0} + \frac 12 \|\bar z_1 - z\*\|^2_{\theta_0^2\gamma_0^{-1}-\beta_0\theta_0}\\
&\leq \frac 12 \|\bar z_0 - z^*\|^2_{\gamma_0^{-1}+\beta_0}=\frac 12 \|z_0 - z^*\|^2_{\gamma_0^{-1}+\beta_0}
\end{align*}
so that
\begin{align*}
G_{\beta_{K+1}}(z_{K+1}) &\leq \frac{\beta_{x,K-1}}{\beta_{x,K}}\mathcal L_{K+1} \leq \frac{K+b}{K-1+b}\frac{b(K/c+1)^{\bar c}}{K+b} \frac 12 \| z_0 - z^*\|^2_{\gamma_0^{-1}+\beta_0} \\
&\leq \frac{b(K/c+1)^{\bar c}}{K+b-1} \frac 12 \| z_0 - z^*\|^2_{\gamma_0^{-1}+\beta_0}
\end{align*}
Finally, since $c \geq 1$, we have $K/c+1 \leq K+1 \leq K+b-1$.
\end{proof}

\section{Restarted accelerated proximal gradient}

The third algorithm we present here is an adaptive restart of Algorithm~\ref{alg:acc_norestart}. When restarting, we set $\theta_k$ back to 1, $\beta_k$ back to $\beta_0$ and $\bar z_{k} = z_k$. Notably, the adaptive restart test is written directly in terms of the optimization objective: indeed, in the present case, we know that the optimal value is 0 and we can use this fact to our advantage.

\begin{algorithm}
	\begin{algorithmic}
	\STATE	$\bar z_0 = z_0 \in \mathcal Z$ \\
	\STATE $s = 0$, $k_s = 0$
	\FORALL{$k \in \mathbb N$}
	\WHILE{$G_{\beta_0}(z_k) > 2^{-s-1} G_{\beta_{0}}(z_0)$}
	\STATE $\hat z_{k} = (1-\theta_{k-k_s}) z_k + \theta_{k-k_s}\bar z_k$
	\STATE $\bar z_{k+1} = \prox_{\frac{\gamma_{k-k_s}}{\theta_{k-k_s}}, F}\Big(\bar z_k - \frac{\gamma_{k-k_s}}{\theta_{k-k_s}}.\nabla F^*_{\beta_{k-k_s}, M}(\hat z_k) \Big)$
	\STATE $z_{k+1} = (1-\theta_{k-k_s})z_k + \theta_{k-k_s} \bar z_{k+1}$
	\ENDWHILE
	\STATE $\bar z_k \leftarrow z_k$, $k_s \leftarrow k$, $s \leftarrow s+1$
	\ENDFOR
	\end{algorithmic}
	\caption{Accelerated proximal gradient descent for the smoothed gap with restart} \label{alg:acc}
\end{algorithm}

\begin{theorem}
If $G_{\beta_0}$ has an error bound with parameters $p \geq 2$ and $\eta>0$, that is if
$G_{\beta_0}(z) \geq \frac{\eta}{p}\dist(z, \mathcal Z_*)^p$, then
$G_{\beta_0}(z_{K})\leq \epsilon$ after
at most 
\begin{equation*}
K = \begin{cases} \lceil -\log_2(\epsilon) \rceil \times \Big \lceil 3-b + \Big(\dfrac{2b \|\gamma_0^{-1}+\beta_0\|_\infty}{\eta}\Big)^{\frac{1}{1-\bar c}}\Big\rceil & \text{ if } p = 2\\
\lceil - \log_2(\epsilon)\rceil (4-b) + \Big(\dfrac{b p^{2/p} \|\gamma_0^{-1}+\beta_0\|_\infty}{\eta^{2/p}G_{\beta_0}(z_0)^{1-2/p}}\Big)^{\frac{1}{1-\bar c}} \dfrac{4^\frac{1-2/p}{1-\bar c}\epsilon^{-\frac{1-2/p}{1-\bar c}}-1}{2^{\frac{1-2/p}{1-\bar c}}-1} & \text{ if } p > 2
\end{cases}
\end{equation*}
iterations.
\end{theorem}
\begin{proof}
Between restarts number $s$ and $s+1$, we can apply Theorem~\ref{thm:acc}. We just need to slide the indices. 
Thus because $\beta_0 \geq \beta_{k-k_s}$ and $G_{\beta_0}(z) \geq \frac{\eta}{p}\dist(z, \mathcal Z_*)^p$, 
we have for all $k \in \{k_s,\ldots, k_{s+1}\}$, and for $z^*$ being the projection of $z_{k_s}$ onto $\mathcal Z_*$,
\begin{align*}
G_{\beta_0}(z_{k}) &\leq G_{\beta_{k-k_s}}(z_{k}) \leq \frac{b}{2(k-k_s+b-2)^{1-\bar c}} \|z_{k_s} - z^*\|^2_{\gamma_{0}^{-1}+\beta_{0}}\\
&\leq \big(\frac{p}{\eta}\big)^{\frac 2 p}\|\gamma_0^{-1}+\beta_0\|_\infty\frac{b}{2(k-k_s+b-2)^{1-\bar c}}G_{\beta_0}(z_{k_s})^{\frac 2p}
\end{align*}
A restart occurs as soon as $G_{\beta_0}(z_{k_{s+1}}) \leq 2^{-s-1} G_{\beta_0}(z_0)$. This means that $G_{\beta_0}(z_{k_{s+1}-1}) > 2^{-s-1} G_{\beta_0}(z_0)$ and $G_{\beta_0}(z_{k_{s}}) \leq 2^{-s} G_{\beta_0}(z_0)$. Hence,
\begin{align*}
& \big(\frac{p}{\eta}\big)^{\frac 2 p}\|\gamma_0^{-1}+\beta_0\|_\infty\frac{b}{2(k_{s+1}-1-k_s+b-2)^{1-\bar c}} 2^{-\tfrac{2s}{p}}G_{\beta_0}(z_0)^{\tfrac{2}{p}} \geq 2^{-s-1} G_{\beta_0}(z_0)\\
&\big(\frac{p}{\eta}\big)^{\frac 2p}\|\gamma_0^{-1}+\beta_0\|_\infty\frac{b}{2(k_{s+1}-k_s+b-3)^{1-\bar c}} > 2^{-s-1}2^{\frac {2s}{p}} G_{\beta_0}(z_{0})^{1-\frac 2 p}\\
& k_{s+1}-k_s+b-3 < 2^{\frac{s(1-2/p)}{1-\bar c}} \Big(\frac{b p^{2/p} \|\gamma_0^{-1}+\beta_0\|_\infty}{\eta^{2/p}G_{\beta_0}(z_0)^{1-2/p}}\Big)^{\frac{1}{1-\bar c}}
\end{align*}
\begin{align*}
& k_{s} < s \Big \lceil 3-b + \Big(\frac{2b \|\gamma_0^{-1}+\beta_0\|_\infty}{\eta}\Big)^{\frac{1}{1-\bar c}}\Big\rceil  & \text{ if } p = 2 \\
& k_s < s (4-b) + \Big(\frac{b p^{2/p} \|\gamma_0^{-1}+\beta_0\|_\infty}{\eta^{2/p}G_{\beta_0}(z_0)^{1-2/p}}\Big)^{\frac{1}{1-\bar c}} \frac{2^{(s+1)\frac{1-2/p}{1-\bar c}}-1}{2^{\frac{1-2/p}{1-\bar c}}-1} & \text{ if } p > 2
\end{align*}
The $4-b$ is here to take care of integers.

Hence, to find a point $z$ such that $G_{\beta_0}(z)\leq \epsilon = 0.5^{-\log_2(\epsilon)}$, we need at most $-\log_2(\epsilon)$ restarts and thus at most $k_{\lceil -\log_2(\epsilon)\rceil}$
iterations.
We can then simplify the last power of $s$ using $2^{\lceil - \log_2(\epsilon) + 1\rceil c} \leq 2^{(- \log_2(\epsilon) + 2)c} = 4^c \epsilon^{-c}$.

\end{proof}

\section{Numerical experiments}

\subsection{Toy linear program}

In the first experiment, we compare the algorithms developed in this paper (denoted as prox\_grad and acc\_prox\_grad in the legend) with PDHG \cite{chambolle2011first} (also know as the Chambolle-Pock algorithm) and its averaged and restarted version RA-PDHG \cite{fercoq2023quadratic}. This choice is motivated by the fact that all these algorithms use exactly the same primitives: proximal operators of the convex functions at stake and matrix-vector multiplications. We also tested an algorithm inspired by L-BFGS:
after each restart, we do a complete smoothing of the duality gap and run L-BFGS on the function
\[
G_{\beta, \delta}(z) := F_\delta(z; -\nabla F^*_{\beta,M}(z_k)) + F^*_{\beta,M}(z)
\]
where $\delta > 0$ and $F_\delta(z; \dot z) = \max_{z'} \langle z, z' \rangle - F^*(z') - \frac \delta 2 \|z' - \dot z\|^2$. This is  acc\_prox\_grad\_lbfgs in the legend. We can see that Proximal gradient on the smoothed gap (Algorithm~\ref{alg:prox_grad_smoothed_gap}) has a similar performance than PDHG while its accelerated and restarted version (Algorithm~\ref{alg:acc}) behaves like RAPDHG. Even if we have no proof for it, the L-BFGS extension presented aboves looks promising on this toy problem.

\begin{figure}
	\centering
	\includegraphics[width=0.7\linewidth]{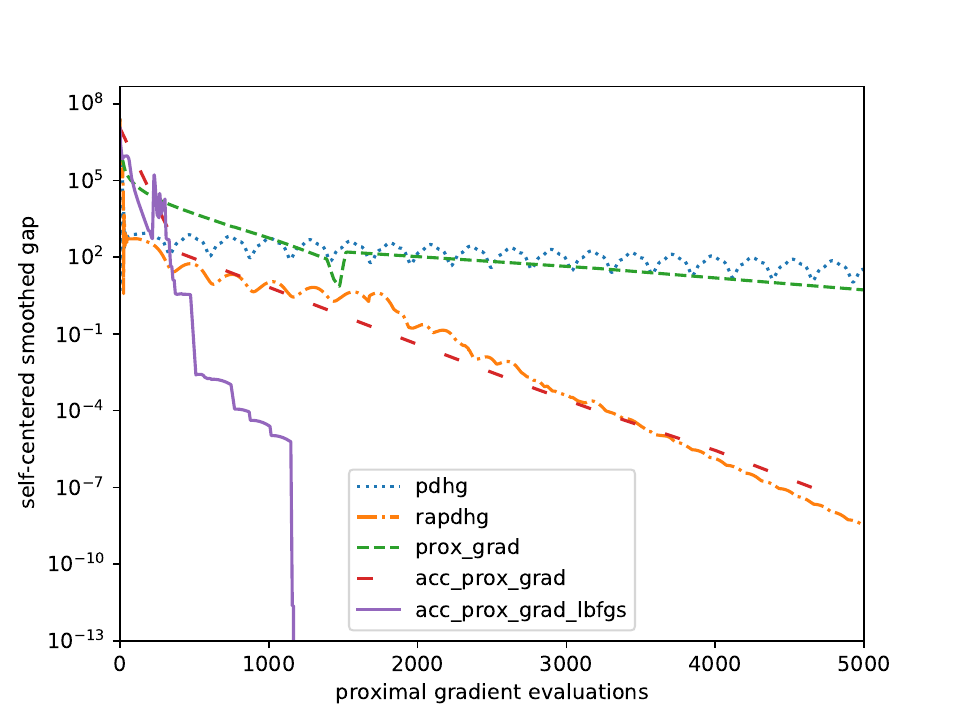}
	\caption{Comparison of algorithms on a toy linear program with 4 variables and 3 constraints}
	\label{fig:expe_toylp}
\end{figure}

\subsection{Second order cone program}

\begin{figure}
	\centering
	\includegraphics[width=0.7\linewidth]{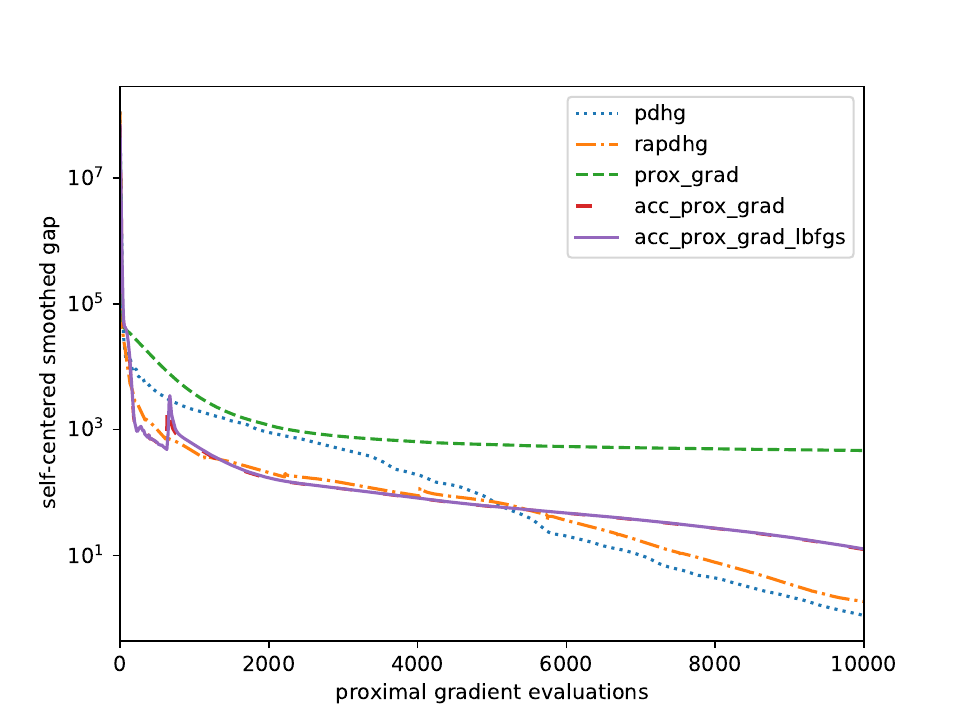}
	\caption{Comparison of algorithms on a second order cone program}
	\label{fig:expe_socp}
\end{figure}

To test our algorithms on a more realistic situation, we considered the
\texttt{qssp30} problem from \texttt{cblib.zib.de} \cite{friberg2016cblib}.
It is a second order cone program with 7,565 variables, 3,691 equality constraints and 1,891 3D second order cone constraints. We can see on Figure~\ref{fig:expe_socp} that on this problem, Algorithm~\ref{alg:prox_grad_smoothed_gap} makes very little progress. Yet, 
Algorithm~\ref{alg:acc} works better, and has similar performance than RAPDHG.
In that case, the L-BFGS extension did not add anything and both curves are nearly overlapping.

\section{Conclusion}

In this paper, we have shown how to reduce the search for a saddle point of a convex-concave function into the minimization of a weakly convex function. We then proposed three algorithms for this minimization problem, inspired by proximal gradient type methods. The algorithms we obtain have a similar theoretical worst case and practical performance than famous algorithms using the same primitives, namely Primal-Dual Hybrid Gradient and its restarted and averaged variant.

This work opens the route for the development of new primal-dual algorithms that use tools from function minimization and extend them to saddle point problems thanks to the self-centered smoothed duality gap. We showed a preliminary experiment on L-BFGS but other ideas include coordinate descent or line search. Another avenue of research could be non-convex non-concave saddle point problems where having a minimization problem to base on can help avoiding pathological situations like limit cycles~\cite{pethick2022escaping}.

\section*{Acknowledgment}

This work was supported by the Agence National de la Recherche grant ANR-20-CE40-0027, Optimal Primal-Dual Algorithms (APDO).

\section*{Declaration of AI-assisted technologies in the manuscript preparation process}
	
During the preparation of this work the author used ChatGPT in order to reformulate some sentences in the introduction, as well as a spell checker. After using these tools, the author reviewed and edited the content as needed and takes full responsibility for the content of the published article.



\bibliographystyle{alpha}
\bibliography{../proxgrad_on_smoothed_gap/literature}

\newcommand{\etalchar}[1]{$^{#1}$}
\begin{thebibliography}{TDAFC20}

\bibitem[AHLL23]{applegate2023faster}
David Applegate, Oliver Hinder, Haihao Lu, and Miles Lubin.
\newblock Faster first-order primal-dual methods for linear programming using
  restarts and sharpness.
\newblock {\em Mathematical Programming}, 201(1):133--184, 2023.

\bibitem[AHU58]{arrow1959studies}
Kenneth~J. Arrow, Leonid Hurwicz, and Hirofumi Uzawa.
\newblock {\em Studies in linear and non-linear programming}.
\newblock Stanford University Press, 1958.

\bibitem[ATDFC17]{alacaoglu2017smooth}
Ahmet Alacaoglu, Quoc Tran~Dinh, Olivier Fercoq, and Volkan Cevher.
\newblock Smooth primal-dual coordinate descent algorithms for nonsmooth convex
  optimization.
\newblock {\em Advances in Neural Information Processing Systems}, 30, 2017.

\bibitem[CP11]{chambolle2011first}
Antonin Chambolle and Thomas Pock.
\newblock A first-order primal-dual algorithm for convex problems with
  applications to imaging.
\newblock {\em Journal of mathematical imaging and vision}, 40:120--145, 2011.

\bibitem[Fer23]{fercoq2023quadratic}
Olivier Fercoq.
\newblock Quadratic error bound of the smoothed gap and the restarted averaged
  primal-dual hybrid gradient.
\newblock {\em Open Journal of Mathematical Optimization}, 4:1--34, 2023.

\bibitem[FR17]{fercoq2017smooth}
Olivier Fercoq and Peter Richt{\'a}rik.
\newblock Smooth minimization of nonsmooth functions with parallel coordinate
  descent methods.
\newblock In {\em Modeling and Optimization: Theory and Applications}, pages
  57--96. Springer, 2017.

\bibitem[Fri16]{friberg2016cblib}
Henrik~A Friberg.
\newblock Cblib 2014: a benchmark library for conic mixed-integer and
  continuous optimization.
\newblock {\em Mathematical Programming Computation}, 8(2):191--214, 2016.

\bibitem[GM76]{gabay1976dual}
Daniel Gabay and Bertrand Mercier.
\newblock A dual algorithm for the solution of nonlinear variational problems
  via finite element approximation.
\newblock {\em Computers \& mathematics with applications}, 2(1):17--40, 1976.

\bibitem[Hes69]{hestenes1969multiplier}
Magnus~R Hestenes.
\newblock Multiplier and gradient methods.
\newblock {\em Journal of optimization theory and applications}, 4(5):303--320,
  1969.

\bibitem[Kor76]{korpelevich1976extragradient}
Galina~M Korpelevich.
\newblock The extragradient method for finding saddle points and other
  problems.
\newblock {\em Matecon}, 12:747--756, 1976.
\newblock Translated from Russian.

\bibitem[Nes05]{nesterov2005smooth}
Yu~Nesterov.
\newblock Smooth minimization of non-smooth functions.
\newblock {\em Mathematical programming}, 103:127--152, 2005.

\bibitem[OV20]{oconnor2020equivalence}
Daniel O’Connor and Lieven Vandenberghe.
\newblock On the equivalence of the primal-dual hybrid gradient method and
  douglas--rachford splitting.
\newblock {\em Mathematical Programming}, 179(1):85--108, 2020.

\bibitem[PLP{\etalchar{+}}22]{pethick2022escaping}
Thomas Pethick, Puya Latafat, Panos Patrinos, Olivier Fercoq, and Volkan
  Cevher.
\newblock Escaping limit cycles: Global convergence for constrained
  nonconvex-nonconcave minimax problems.
\newblock In {\em International Conference on Learning Representations}, 2022.

\bibitem[Pow78]{powell1978algorithms}
Michael~JD Powell.
\newblock Algorithms for nonlinear constraints that use lagrangian functions.
\newblock {\em Mathematical programming}, 14(1):224--248, 1978.

\bibitem[Roc76]{rockafellar1976augmented}
R~Tyrrell Rockafellar.
\newblock Augmented lagrangians and applications of the proximal point
  algorithm in convex programming.
\newblock {\em Mathematics of operations research}, 1(2):97--116, 1976.

\bibitem[TDAFC20]{tran2020adaptive}
Quoc Tran-Dinh, Ahmet Alacaoglu, Olivier Fercoq, and Volkan Cevher.
\newblock An adaptive primal-dual framework for nonsmooth convex minimization.
\newblock {\em Mathematical Programming Computation}, 12(3):451--491, 2020.

\bibitem[TDFC18]{tran2018smooth}
Quoc Tran-Dinh, Olivier Fercoq, and Volkan Cevher.
\newblock A smooth primal-dual optimization framework for nonsmooth composite
  convex minimization.
\newblock {\em SIAM Journal on Optimization}, 28(1):96--134, 2018.

\bibitem[Tse08]{tseng2008accelerated}
Paul Tseng.
\newblock On accelerated proximal gradient methods for convex-concave
  optimization.
\newblock {\em submitted to SIAM Journal on Optimization}, 2(3), 2008.

\bibitem[WF25]{walwil2025primal}
Iyad Walwil and Olivier Fercoq.
\newblock Primal-dual coordinate descent for nonconvex-nonconcave saddle point
  problems under the weak mvi assumption.
\newblock {\em arXiv preprint arXiv:2506.15597}, 2025.

\end{thebibliography}

\end{document}